\documentclass[11pt]{article}
\usepackage{amscd,amssymb,amsmath,latexsym,graphicx,amsxtra,amsthm}
\usepackage{makeidx}
\usepackage[all]{xy}
\usepackage{graphicx}
\usepackage{epstopdf}

\def\ndv{\ {\mid \kern -0.65 em {\scriptstyle \not}}\ \ \ }

\def\Z{\Bbb{Z}}
\def\R{\Bbb{R}}

\def\G{\mathcal{G}}
\def\E{\mathcal{E}}

\def\bb{{\vec b}}

\def\int{\operatorname{int}}

\def\Int{\operatorname{Int}}

\def\ovl{\overline}

\newtheorem{theorem}{Theorem}[section]
\newtheorem{lemma}[theorem]{Lemma}
\newtheorem{corol}[theorem]{Corollary}
\newtheorem{remark}[theorem]{Remark}
\newtheorem{prop}[theorem]{Proposition}

\newtheorem{obs}{Remark}[section]
 \newtheorem{dfn}{Definition}[section]

\usepackage{color}
\usepackage{color}
 \long\def\red#1{\textcolor {red}{#1}}

\title{{\bf Deficient and multiple points of maps into a manifold}}

\author{Daciberg  L. Gon\c calves \\ Tha\'is F. M. Monis \\ Stanis\l{a}w Spie\.z}

\begin{document}

\maketitle

\abstract{For a map $f:X \to M$ into a manifold $M$, we study the sets of deficient and multiple points of $f$. In case of the set of deficient points, we estimate its dimension. For multiple points, we study its density in $X$, and  we also provide examples where its complement is dense.


\section{Introduction}

P.~T.~Church and J.~G.~Timourian in \cite{CT} studied the deficient points of maps between manifolds. 
\red{For} a map $f: M \to N$ between manifolds of same dimension, a point $y\in N$ is called deficient if the pre-image of
$y$ has less then $A(f)$ points, where $A(f)$ denotes the Hopf's absolute degree of $f$. Related to the notion of deficient point, there is the notion of multiple point of a map. A point $x \in X$ is a multiple point of a map $f: X \to Y$ if $f^{-1}(f(x))\neq \{x\}$, otherwise $x$ is called a single point. The present work extends some results about deficient and multiple points in the literature, which we will describe in more detail.

By {\it dimension of a space $X$ we mean the  covering dimension}, 
which sometimes is called  topological dimension (see  \cite{HW} and \cite{Engelking}). The following result was proved by P.~T.~Church and J.~G.~Timourian in \cite{CT}. 

\begin{theorem}[Theorem 1.1, 2.7, 3.3 and 3.4 \cite{CT}]  \label{theoremCT} 
Suppose $M$ and $N$ are oriented connected $n$-manifolds and $f: M \to N$ is a proper map  with degree $\deg f \ne 0$.
Let $E_f$ be the set of points $y \in N$ for which $f^{-1}(y)$ is discrete and has less than $|\deg f|$ essential points. 
\begin{enumerate}

\item[\rm(1)] Then $\dim E_f \le n-1$ and, moreover, $E_f$ contains no closed (in $N$) subset of dimension $n-1$. 

\item[\rm(2)] If $f$ is discrete (i.e., each $f^{-1}(y)$ is discrete), then $\dim (\ovl{E}_f) \le n-2$. 

\end{enumerate}
\end{theorem} 

In the same paper, the authors generalized this theorem to manifolds which are not necessarily orientable replacing 
$|\deg f|$ by Hopf's absolute degree $A(f)$ (c.f., 5.13 \cite{CT}). 

We extend the notion of absolute degree and essentially deficient points for maps $f: X \to M$ into manifolds, with $X$ not necessarily a manifold. 
Then we prove a result, Theorem 2.11, which extends Theorem \ref{theoremCT}, where $X$ \red{no longer is assumed} to be  a manifold but it satisfies certain cohomological property.

In a study about multiple points of maps between manifolds,  D. L. Gon\c{c}alves proved the following.

\begin{theorem}[\cite{GonOre}, Theorem 2.2]  Let $M,N$ be manifolds of same dimension, $M$ closed. Let $[f]$  be the    homotopy class of  a    map $f:M \to N$.
\begin{itemize}
 \item[a$)$] If the absolute  degree of the map  $f$ is  $1$, then there is a   map $g\in [f]$  such 
 that the set of  multiple   points  is not dense in $M$.
 \item[b$)$]  If $A(f)\ne1$ then the set of multiple points of $f$ is dense.
  \end{itemize}

\end{theorem}

The above theorem follows from properties of the Hopf's absolute degree found in the paper of Epstein \cite{Epstein}. Namely, item (a) follows from Hopf's theorem [\cite{Epstein}, Theorem 4.1] and item (b) follows from [\cite{Epstein}, Lemma 2.1a and Theorem 3.1].  We prove Theorem \ref{dominating} which generalizes the item (b) of the above theorem. In this theorem we consider maps from $X$ into $M$ for  certain CW-complex $X$, which are not necessarily manifold.
We also prove another result, Theorem \ref{theorem1}, for a different family of CW-complex, under different hypotheses. A family of complexes which are not manifolds satisfying hypotheses of Theorem \ref{theorem1} is explicitly given.

S. Orevkov gave in  \cite{GonOre} an example of a map $f: S^2 \to \R^2$ for which the set of single points (the complement of the set of multiple points) is dense in $S^2$.  We mimic his example   to show that in many cases one can construct maps where the set of single points is large, namely it is dense. See   Corollary \ref{corol4.6} and Remark 4.6.

The present work contains 3 sections besides the introduction. In section 2,  we estimate the dimension of the set of essentially deficient points. The main results of the section are Propositions 2.8, \ref{prop2} and Theorem \ref{maintheorem}.  Then, in section 3, we consider the study of multiple points, as well single points.  In this case, the results deal with the existence of maps where the set of multiple points is dense.  The main results are Theorem 3.1 and Theorem 3.3. Finally,  in section 4, we show that in many cases one can construct maps where the set of single points is dense. The main results are  Theorem 4.2 and Corollary 4.5.







\section{Deficient points}

In this section, we assume $X$ a connected,  locally path-connected, locally compact  $n$-dimensional paracompact Hausdorff space, 
$M$ an $n$-manifold  without boundary and $f: X \to M$  a proper map. 
By  $\G$ we denote  the orientation system (over $\Z$) on $M$ and by  $\G_f$  the local system on $X$  induced by $\G$ and $f$. 
We consider \u{C}ech cohomology groups with local coefficients and with compact carriers. 

\subsection{Local degree, essential points}

Suppose  that $x\in X$  is an isolated point of the set $f^{-1}(f(x))$,  
i.e., $x \notin \ovl{f^{-1}(f(x))\setminus \{x\}}$. 
Then there is a compact set $V \subset X$  such that  $x \in \int V$  and $V \cap f^{-1}(f(x)) = \{x\}$.

Consider the following diagram 

\begin{table}[h]
\centering
\begin{tabular}{c} \xymatrix
{
 H^{n} (X;\G_f)  &  H^n(X, X \setminus \{x\};\G_f) \ar[l]_{\hspace{-.6cm}H^{n}(i)} \ar[r]^{\hspace{-.2cm}H^{n}(j)} & 
H^n(V,V \setminus \{x\};\G_f |V)  \\
  & H^n(M;\G)   & \ar[l]^{\hspace{-.6cm}H^{n}(i')} H^n(M, M \setminus \{f(x)\};\G)  \ar[u]^{H^n(f|V)} 
}
\end{tabular}
\end{table}

\noindent 
where  $i$ and $i'$ are the inclusions, $f|V$ is the restriction of $f$ and $H^n(j)$ is the excision isomorphism.
Note that the composition $(H^n(j))^{-1} \circ H^n(f|V)$ does not depend on the choice of $V$.

Let  $g_M$ be a generator of $H^n(M;\G)$.
 The image 
$$ \text{d}_x f : =  \big( H^n(i) \circ (H^n(j))^{-1} \circ H^n(f|V) \circ  (H^n(i'))^{-1} \big) (g_M) \in  H^{n} (X;\G_f) $$
we call {\it the local twisted degree element} of $f$ at $x$. 

 We say that $x$ is {\it an essential point} of $f$  if  $\text{d}_x f$ is non-zero.
The set of essential points of $f$ is denoted by $X_f$.  

The following two lemmas are consequences of the excision and the Mayer-Vietoris sequence.

\begin{lemma}[cf., Lemma 2.3 \cite{CT}] \label{degree} 
Suppose  for some $y \in M$,  $f^{-1}(y)$ is discrete (so finite  since $f$ is proper).
Then
$$H^n(f)(g_M) = \sum_{x \in  f^{-1}(y)}  \text{d}_x f  =  \sum_{x \in  f^{-1}(y) \cap X_f }  \text{d}_x f   \ . $$
In particular, if $f^{-1}(y) \cap X_f = \{$x$\}$ then $H^n(f)(g_M) =  \text{d}_x f$. 
\end{lemma}

\begin{lemma}\label{local_degree}
Let $B$ be an $n$-ball in $M$ with $y \in \int B$, and $V$ a component of $f^{-1}(\int B)$ such that $f^{-1}(y)\cap V=\{x\}$.
Then, for any $z \in \int B$ with  $f^{-1}(z)\cap V$ finite  we have 
$$\text{d}_x f = \sum_{x' \in  f^{-1}(z) \cap V}  \text{d}_{x'} f  = 
 \sum_{x' \in  f^{-1}(z) \cap V \cap X_f }  \text{d}_{x'} f   $$ 
and, consequently,  $f^{-1}(z) \cap V \cap X_f \ne \emptyset$ if  $x$ is essential. 
\end{lemma}
\begin{proof}
Let  us consider the following diagram:

\begin{table}[h]
\centering
\begin{tabular}{c} \xymatrix
{
 H^{n} (X;\G_f)  &  H^n(X, X \setminus V;\G_f) \ar[l]_{\hspace{-.6cm}H^{n}(i_V)} \ar[r]^{\hspace{-.2cm}H^{n}(j_V)} & 
H^n(\ovl{V},\ovl{V} \setminus V;\G_f|\ovl{V}) \\
  & H^n(M;\G)   & \ar[l]^{\hspace{-.6cm}H^{n}(i_B)} H^n(M, M \setminus \int B;\G)  \ar[u]^{H^n(f|\ovl{V})}  \, ,
}
\end{tabular}
\end{table}

\noindent 
where  $H^{n} (i_V))$  and  $ H^{n}(i_B)$  are the homomorphisms induced by the inclusions, 
$H^{n} (j_V)$ is the excision isomorphism and $H^n(f|\ovl{V})$ is the homeomorphism induced by 
the restriction $f|\ovl{V}$ of $f$. 
One can show that 
$$ 
\text{d}_x f =  \big( H^n(i_V) \circ (H^n(j_V))^{-1} \circ H^n(f|\ovl{V}) \circ  (H^n(i_B))^{-1} \big) (g_M).
$$
The assertion of the lemma is  a consequence of the above equality. 
\end{proof}



\begin{remark}\label{rem_cov} If  $p: \widetilde{M} \to M$ is a covering 
and $\widetilde{f} : X \to  \widetilde{M}$ a lifting of $f$ then 
\begin{enumerate}
\item[{\rm(a)}] $X_f=X_ {\widetilde{f}}$,
\item[{\rm(b)}] the local system $\G_p$  induced by $\G$ and $p$ is an orientation  system  on $\widetilde{M}$,  
\item[{\rm(c)}] the local system  induced by  $\G_p$ and $\widetilde{f}$  
  coincides with the local system $\G_f$, and 
\item[{\rm(d)}] if $x \in X$ is an isolated point of $f^{-1}(f(x))$
then it is an  isolated point of $\widetilde{f}^{-1}(\widetilde{f}(x)) \subset X$ and 
$\text{d}_x f = \text{d}_x \widetilde{f}$.
\end{enumerate}
\end{remark}

\subsection{Essentially deficient points}

We say that a point  $y \in M$  is {\it essentially $k$-deficient}  if $f^{-1}(y)$ is discrete and has fewer than $k$ essential points.
The set of all such points we denote by $E_{k,f}$, i.e., 
$$E_{k,f} = \{  y \in M : f^{-1}(y)\ {\rm  is \  discrete \ and} \  |f^{-1}(y) \cap X_f| <k \}. $$  

\begin{remark}\label{max}
Let $T$ be an open subset in $M$ and suppose that $f^{-1}(y)$ is discrete for some $y \in T$.
 Then there exists an $n$-ball $B \subset T$ with $y \in \int B$ such that each component  of $f^{-1}(\int B)$ 
 contains at most one point of $f^{-1}(y)$. 
In this case,   the   component  of $f^{-1}(\int B)$  containing $x\in f^{-1}(y)$  we denote  by $V_x$. 
\end{remark}

The following lemma  is a modification of  Lemma 2.6 \cite{CT} and its proof  follows similar arguments.  

\begin{lemma}\label{main}
Let $T$ be open in $M$, and let $A \subset E_{k,f} \cap T$ be nonempty. 
Then there exists $y \in A$ such that, with the $B \subset T$ and $\{V_x\}_ {x \in f^{-1}(y)}$  as in Remark \ref{max},  
we have:
\begin{enumerate}
\item[{\rm(1)}] $ X_f \cap f^{-1} (A \cap \int B ) \subset  \bigcup_{x \in f^{-1}(y) \cap X_f} V_x$, 
\end{enumerate}
and for each $x \in f^{-1}(y) \cap X_f$
\begin{enumerate}
\item[{\rm(2)}]  the map $f_x : V_x \cap  X_f \cap f^{-1}(A) \to A \cap \int  B $ induced by $f$ 
is a homeomorphism, and
\item[{\rm(3)}] $V_x \cap  X_f \cap f^{-1}(A) $  is closed in $f^{-1}(A \cap \int B)$.   
\end{enumerate}
\end{lemma}

\begin{proof} 
 For each $y \in A$, $ f^{-1}(y)$ is finite and has $m(y)$ 
essential points, where $m(y) < k$. 
Choose $y$ such that $m(y)$ is maximal. 
Let $B \subset T$ and $ \{V_x\}_{x \in f^{-1}(y)}$  be as given by \ref{max}. 
Then, by Lemma \ref{local_degree}, for each  $x \in f^{-1}(y) \cap X_f$ and for each $z \in A \cap \int B$, $f^{-1}(z)  \cap V_x$ has at least one essential point.  
Thus, by the maximality property of $y$, we obtain 
\begin{enumerate}
\item[{\rm (a)}] if  $x \in f^{-1}(y) \cap X_f$  and  $z \in A \cap \int B$  then $f^{-1}(z)  \cap V_x$ 
has exactly one essential point, and 
\item[{\rm (b)}]  if  $x \in f^{-1}(y) \setminus X_f$ then $V_x \cap X_f = \emptyset$.
\end{enumerate}
Observe that (b) implies (1). 

For each $x \in f^{-1}(y) \cap X_f$,   the function  $f_x : V_x \cap  X_f \cap f^{-1}(A) \to  A \cap \int  B$ in (2) 
is continuous and by (a),  bijective . 
Now, we show that  its inverse $g_x$ is continuous. 
Take $z \in A \cap \int B$  and an arbitrary open neighborhood $W$ of $g_x(z)$ in $V_x$.
Then there exists an open $n$-ball  $D \subset \int B$ with $z \in  D$ such that the component $V$ of $f^{-1}( D)$ 
containing $g_x(z)$  is contained in $W$. 
Since $g_x(z) \in X_f$, then one can observe that the homomorphism 
$$
  H^n(V, V \setminus \{g_x(z)\};\G_f|V) \leftarrow  H^n(D, D \setminus \{z\};\G|D)
$$
induced by $f|V$  is nontrivial.  Thus $f(V)=\int D$.  
Thus $g_x(A \cap \int D) = V \cap X_f \cap f^{-1}(A) \subset  W \cap X_f \cap f^{-1}(A)$. 
It follows that $g_x$ is continuous, which proves (2).  
 
By (2) there is a retraction of $f^{-1}(A \cap \int B)$ onto   $V_x \cap  X_f \cap f^{-1}(A)$, so (3) follows. 
\end{proof}

\begin{corol}\label{cor1}
Suppose that $\dim E_{k,f}=n$ and let $T$ be an open nonempty  subset of $M$ contained in $ E_{k,f}$. 
Then there exists $y \in T$ such that, with the $B \subset T$ and $\{V_x\}_ {x\in f^{-1}(y)}$  as in Remark \ref{max}, 
we have:

\begin{enumerate}
\item[{\rm(1)}] $ X_f \cap f^{-1}  ( \int B ) \subset  \bigcup_{x \in f^{-1}(y) \cap X_f} V_x$,
\end{enumerate}
and for each $x \in f^{-1}(y) \cap X_f$
\begin{enumerate}
\item[{\rm(2)}]  $f_x : V_x \cap  X_f  \to \int  B $ induced by $f$ is a homeomorphism, and 
\item[{\rm(3)}]  $V_x \cap  X_f $  is closed in $V_x$. 
\end{enumerate}
\end{corol}

\begin{lemma}[cf., Lemma 2.5 \cite{CT}] \label{discrete}
If $f$ is discrete then $E_{k,f}$ is closed.  
\end{lemma}
\begin{proof}  Take an arbitrary $y \in M \setminus E_{k,f}$. 
Since $f$ is discrete, there is an $n$-ball $B$ in $M$ with $y \in \int B$ and such that each component of $f^{-1} (\int B)$ contains at most one point of $f^{-1}(y)$.
By $V_x$ we denote the component of $f^{-1}(\int B)$ containing $x\in f^{-1}(y)$. 
By Lemma \ref{local_degree}, given $z\in \int B$, for each essential point $x\in f^{-1}(y)$, $f^{-1}(z)\cap V_x \cap X_f \neq \emptyset$. 
Hence, for each $z\in \int B$, $|f^{-1}(z) \cap X_f| \ge |f^{-1}(y) \cap X_f| \ge k$, so  $\int B \subset M \setminus E_{k,f}$. 
Thus  $ M \setminus E_{k,f}$ is open, which proves the lemma. 
\end{proof}

\subsection{Twisted degree} \label{twist}

From now, we additionally assume  that  $H^n(X;\G_f)$ is either an infinite cyclic group or a cyclic group of order 2. 

Let  $g{_X}$ be a fixed generator of $H^n(X;\G_f)$.
Then {\it the twisted degree} of $f$  is the integer $\deg f$ such that 
$$(H^n(f))(g_M) =( \deg f) g_X  ,$$ where in case  $H^n(X;\G_f)$ is the cyclic group of order 2, we let $\deg f$ be $0$ or $1$.

For $x \in X$, let  $k_x$ be the index  of the image of  $H^{n}(i)$ in the group $ H^{n} (X;\G_f)$,  
where  $H^{n}(i): H^n(X, X \setminus \{x\};\G_f) \to H^{n} (X;\G_f)$ is the homomorphism induced by the inclusion,  if such homomorphism is non-trivial; otherwise, let $k_x=0$.

Then, by $k_f$ we denote  the maximum of $k_x$  over all   $x \in X$.
\medskip


We also assume that $p: \widetilde{M} \to M$ is a covering  with a finite number $j$ of sheets and 
that $\widetilde{f} : X \to  \widetilde{M}$ is a lifting of $f$.  

\begin{lemma}\label{lemma_cov}
Suppose  $k_f>0$ and let $y \in E_{k,f}$, where  $k = j \cdot |\deg {\widetilde{f}}| / k_f$. 
Then there exists $x \in f^{-1}(y) \cap X_f$ such that $|\text{d}_x f| > k_f$. 
\end{lemma}

\begin{proof}
Let $p^{-1}(y)=\{ {\widetilde{y}}_1, \cdots,  {\widetilde{y}}_j\}$. 
Then ($\ast$) 
$ 
f^{-1}(y)= \bigcup_{i\in\{1, \ldots , j\}}   {\widetilde{f}}^{-1}( {\widetilde{y}}_i)  
$ 
and each  ${\widetilde{f}}^{-1}( {\widetilde{y}}_i)$ is finite as $ f^{-1}(y)$ is finite.
\smallskip

By Lemma \ref{degree}, for each $i \in \{1, \ldots , j\}$ we have
$$
|\deg   \widetilde{f}|= 
\Big|\sum_{x \in  \widetilde{f}^{-1}( {\widetilde{y}}_i) \cap X_{{\widetilde{f}} } }  \text{d}_x {\widetilde{f}} \Big|  \le 
  \sum_{x \in  {{\widetilde{f}}}^{-1}( {\widetilde{y}}_i) \cap X_{{\widetilde{f}}} }  |\text{d}_x {\widetilde{f}}| \ . 
$$
Thus, by ($\ast$) and  Remark \ref{rem_cov} (a),  
$ j \cdot |\deg   {\widetilde{f}}| \le 
  \sum_{x \in  f^{-1}( y) \cap X_f}  |\text{d}_x {\widetilde{f}}|$. 
Consequently, as  $k = j \cdot |\deg {\widetilde{f}}| / k_f$ and $y \in E_{k,f}$, we obtain   
$$
 |f^{-1}( y) \cap X_f| \cdot  k_f  < 
  \sum_{x \in  f^{-1}( y) \cap X_f}  |\text{d}_x {\widetilde{f}}| \ . 
$$
So, $ k_f < |\text{d}_x {\widetilde{f}}| $ for some $ x \in  f^{-1}( y) \cap X_f $ and the conclusion follows since $d_xf = d_x \tilde{f}$ by Remark \ref{rem_cov} (d).
\end{proof}

\subsection{Dimensional properties of the set $E_{k,f}$ }

\begin{dfn} We say that a point of  a space  $X$ is $n${\it -Euclidean} if it has a neighborhood in $X$ homeomorphic to $\R^n$. By $\mathcal{E}_n(X)$ we denote the set of all $n$-Euclidean points of $X$.
Note that  $\E_n(X)$ is an open subset  of $X$. 

\end{dfn}

Let $k_f$ be positive  and  let $k := j \cdot |\deg {\widetilde{f}}| / k_f$.

\begin{prop} [cf., Proposition 2.7 \cite{CT}]\label{prop1}
 If $\dim (X \setminus \E_n(X)) < n$ then  $\text{dim} \, E_{k,f} \leq n-1$.
\end{prop}

\begin{proof}
Suppose $\text{dim} \, E_{k,f} =n$. Then there is an open subset $T$ of $M$ contained in $E_{k,f}$. 
Let $y \in T$,  $B \subset T$ and $\{V_x\}_{x \in f^{-1}(y)}$ be given as in  
Corollary \ref{cor1}. 

By (2) \ref{cor1}, for each $x \in f^{-1}(y) \cap X_f$, $f$ maps homeomorphically $V_x \cap X_f$ onto the open $n$-ball $\int B$. 
Since  $\dim (X \setminus \E_n(X)) < n$,  it follows that $V_x \cap X_f \cap \E_n(X)$ 
is a non-empty open subset of $V_x \cap X_f \approx \int B$. 
It follows that any $x'\in V_x \cap X_f \cap \E_n(X)$, we have $|d_{x'}f| \le k_f$. 
Note  also that, $f^{-1}(f(x'))\cap V_x \cap X_f = \{x'\}$,  so $d_x f=d_{x'} f$, by Lemma \ref{local_degree}.
Thus $|d_{x}f| \le k_f$  for each $x \in f^{-1}(y) \cap X_f$, which contradicts  \ref{lemma_cov}. 
So,  $\text{dim} \, E_{k,f} \le n-1$. 
\end{proof}

\begin{prop} \label{prop2}
$ E_{k,f} $ does not contain a closed (in $M$) subset $F$ of dimension $n-1$ and disjoint with $\ovl{f(X \setminus \E_n(X))}$. 
\end{prop}

\begin{proof}
Suppose in contrary, 
$ E_{k,f} $ contains  a closed (in $M$) subset $F$ of dimension $n-1$  disjoint with $\ovl{f(X \setminus \E_n(X))}$. 
Let $S$ be a closed (in $M$) subset of $F$  such that 
\begin{enumerate}
\item[{\rm(a)}] for every open  $W \subset M$ with $S \cap  W \ne \emptyset$,   $\dim (S \cap  W) = n-1$. 
\end{enumerate}

Let $T$ be an  open subset of $M$ with $S \cap  T \ne \emptyset$, and let $A = S \cap  T$.  
By Lemma \ref{main}, there exists $y \in A$ such that,  with the $B \subset T$ and $\{V_x\}_ {x \in f^{-1}(y)}$  as in \ref{max},  
for each $x \in f^{-1}(y) \cap X_f$ the conditions (1), (2) and (3) in  \ref{main} are satisfied. 

Since $f^{-1}(A) \subset \E_n(X)$, we may choose the $n$-ball $B$ (with $y \in \int B$) so small that  $f^{-1}(B) \subset \E_n(X)$, 
and moreover, that  each $V_x$  (a component of  $f^{-1}(\int B)$) is a connected  oriented $n$-manifold. 
Thus each  $f_x : V_x \to \int B$, being the restriction of $f$,  is a proper map between oriented connected $n$-manifolds. 

By Lemma \ref{lemma_cov}, there exists $x_0 \in f^{-1}(y) \cap X_f$ such that $|\text{d}_{x_0} f| > k_f$. 
It follows that $|\deg  f_{x_0}| >1$.  
Then, one can observe that $E_{f_{x_0}}$,  as defined in Theorem \ref{theoremCT},  contains $A$, a closed subset of $\int B$ of dimension $n-1$.
This contradicts  Theorem \ref{theoremCT},  which completes the proof. 
\end{proof}

\begin{corol} \label{prop3}
If $f$ is discrete then $\dim \big(E_{k,f} \setminus \ovl{f(X \setminus \E_n(X))}\big) \le n-2$.  
\end{corol}
\begin{proof}  
Let $\{U_i\}_{i=1}^\infty$ be a countable family of open subsets of $M$  such that $\bigcap U_i = \ovl{f(X \setminus \E_n(X))}$.
Each $E_{k,f} \setminus U_i$, by Lemma \ref{discrete}, is  closed, so by Proposition \ref{prop2},  $\dim \big(E_{k,f} \setminus U_i ) \le n-2$. 
Since $ E_{k,f} \setminus \ovl{f(X \setminus \E_n(X))}$  is a union of a countable number of closed subsets $ (E_{k,f} \setminus U_i)$ of dimension $\le n-2$, we obtain that  $\dim \big(E_{k,f} \setminus \ovl{f(X \setminus \E_n(X))}\big) \le n-2$. 
\end{proof}

\subsection{Absolute degree and the main theorem } \label{absolute} Let
 $p: \widetilde{M} \to M$ be the covering  such that there is a lifting of $f: X \to M$ to $\widetilde{f} : X \to \widetilde{M}$ 
and that there is no covering of $\widetilde{M}$ with this property. 
Note that  the homomorphism $\widetilde{f}_{\ast} : \pi_1 X \to \pi_1\widetilde{M}$ is onto. 
Let $j= | \pi_1 M :f_{\ast}\pi_1 X|$. 
Then  $p: \widetilde{M} \to M$ has $j$ sheets.

{\it The absolute degree} $A(f)$ of $f$  is the integer $|\deg \widetilde{f}| \cdot j$ if $j$ is finite 
and it is $0$ if  $j$ is infinite. 
If $X$ is a closed $n$-dimensional manifold  this notion coincides with the Hopf's absolute degree, see \cite{Olum} (16.5) and (17.7). 

A point $y \in M$ is called {\it essentially deficient with respect} to $f$  if $f^{-1}(y)$ is discrete and 
$|f^{-1}(y) \cap X_f| \cdot  k_f <  A(f)$. As in Lemma \ref{lemma_cov} and in Proposition \ref{prop1},  we assume $k_f$ positive.
The set of all essentially deficient  points (with respect to $f$)   is denoted by $E_f$, i.e.,
$$ E_f = \{y \in M :  |f^{-1}(y)| < \infty \ {\rm and} \  |f^{-1}(y) \cap X_f| \cdot  k_f <  A(f) \} \ .$$ 
Observe that if $k_f>0$ then $E_f=E_{k,f}$, where $k$ is the smallest integer greater than or equal to $A(f)/k_f$.

\medskip

By Proposition \ref{prop1},  Proposition \ref{prop2} and  Corollary \ref{prop3} we obtain 

\begin{theorem}[cf., Propostion 5.13 \cite{CT}] \label{maintheorem}
Suppose that $A(f) \ne 0$ and $k_f>0$. Then 
\begin{enumerate}
\item[{\rm(1)}]  If $\dim (X \setminus \E_n(X)) < n$ then  $\text{dim} \, E_f \leq n-1$.
\item[{\rm(2)}] $ E_f $ does not contain a closed (in $M$) subset $F$ of dimension $n-1$ 
                and disjoint with $\ovl{f(X \setminus \E_n(X))}$, and 
\item[{\rm(3)}] If $f$ is discrete then $\dim \big(E_f \setminus \ovl{f(X \setminus \E_n(X))}\big) \le n-2$.  
\end{enumerate}
\end{theorem}

\medskip 

\noindent{\bf Example.} 
Let $X$ be the space obtained by attaching to $S^1$ two $2$-cells, $e_1^2, e_2^2$, using maps of degree $2$ and $3$, respectively.  Let $S^1 \subset S^2$ be the equator of the sphere $S^2$ and let $f: X \to S^2$ given by: on the $1$-skeleton, $S^{1}$, it is the identity.  On $e_1^2$ it covers the north hemisphere and it is given by a map of the type $z \mapsto z^2$, and on $e_2^2$ it covers the south hemisphere and it is a map of the type $z \mapsto z^3$.  
Then, $k_f=3$, $A(f)=6$ and $$E_f = \{ (x,y, 0) \in S^2\} \cup \{ (0,0,1), (0,0,-1)\}.$$ Hence, in this example, $\text{dim} \, E_f=1$. 
This shows that the assumption in  Theorem \ref{maintheorem} (2), that $F$  is disjoint with $\ovl{f(X \setminus \E_n(X))}$, is essential.

\section{Multiple points of maps on  a  complex}

 Recall that a point $x\in X$ is a multiple point of a function $f:X \to Y$ if $f^{-1}(f(x))\neq \{x\}$. We provide two results. The first result is in subsection 3.1, for maps $f: X \to M$ where $X$ and $M$ satisfies the hypotheses of the previous section.  In subsection 3.2, we will consider maps $f: X \to M$ for $X$ a more general space.

\subsection{Multiple points }

 We say that $X$ is {\it locally non-trivial with respect to a local system $\mathcal H$}  over $X$,  if the homomorphism $H^{n}(i): H^n(X, X \setminus \{x\};\mathcal H) \to H^{n} (X;\mathcal H)$ induced by the inclusion  is nontrivial for each $x \in X$. 
\medskip

Let $f:X \to M$, $\G$ and $\G_f$ be as in Section 2. 

\begin{theorem}\label{dominating}
Suppose that $X$  is locally non-trivial with respect to $\G_f$ and that the set of multiple points of $f$ is not dense in $X$.
Then $f_{\ast} : \pi_1 X \to \pi_1 M$ is onto
and $A(f)=|\deg f|$.  Consequently, $0 < A(f) \le k_f $. 
\end{theorem}
\begin{proof} Since $\mathcal{M}(f)$, the set of multiple points of $f$, is not dense in $X$ there exists an open subset  $V$ of $X$ of single points.  
Let $x\in V$. 
Then, there exists a neighborhood $U$ of $f(x)$ such that $f: f^{-1}(U) \to U$ is a homeomorphism. 

It follows that $f^{\ast} : H^n(M, M \setminus \{f(x)\}; \G) \to H^n(X, X \setminus \{x\}; \G_f)$ is an isomorphism. From the assumption that $X$ is locally non-trivial, we have that $H^{n}(i): H^n(X, X \setminus \{x\};\G_f) \to H^{n} (X;\G_f)$ is non-trivial. Thus, by the commutativity  of the diagram 

\begin{table}[h]
\centering
\begin{tabular}{c} \xymatrix{ H^n (X, X \setminus \{x \}; \G_f)   \ar[r]^{\  \  \  \neq 0}& H^n(X; \G_f)  \\
H^n(M, M \setminus \{f(x)\}; \G)  \ar[r]^{\  \  \  \  \ \simeq}   \ar[u]_{\simeq} & H^n(M, \G) \ar[u]}
\end{tabular}
\end{table}
\noindent  it follows that $f^{\ast} : H_n(M;\G) \to H^n(X,\G_f)$ is a non-trivial homomorphism. 
Note also that \\ 
(i)  \ \ \  $  \,   |d_xf| = |k_x| > 0$.


Let $p: \widetilde{M} \to M$ be the covering  such that there is a lifting of $f: X \to M$ to $\widetilde{f} : X \to \widetilde{M}$ 
and that there is no covering of $\widetilde{M}$ with this property. 

Since $f^{\ast}: H^n(M; \G) \to H^n(X, \G_f)$ is non-trivial and $f= p \circ \widetilde{f}$, $\widetilde{f}^{\ast}: H^n(\widetilde{M}; \widetilde{\G}) \to H^n(X; \G_f) $ is non-trivial. Therefore, $\widetilde{f}$ is surjective. 
Thus, since  $f: f^{-1}(U) \to U$  is a bijection, $p$ is the identity. 
Hence,  $f_{\ast} : \pi_1 X \to \pi_1 M$ is onto and \\
(ii) \ \   $A(f)=|\deg f|$. 

By Lemma \ref{degree}, $|\deg f|= |d_x f|$.  Thus, by  (i) and  (ii), we obtain  $0 < A(f) \le k_f $. 
\end{proof}

\begin{remark}  Note that the above theorem generalizes Theorem 2.1 b) of \cite{GonOre} because, in case of $X$ a manifold, $k_f=1$.

\end{remark}

\bigskip


\subsection{Case of  more general complexes}

In this subsection, by $H$ we denote the \u{C}ech cohomology functor with integer coefficients. 

\

\begin{dfn} A space $X$ has dimension $\geq n$ at a point $x\in X$ if each neighborhood of $x$ $($in $X$$)$ has dimension $\geq n$. If, in addition, some neighborhood of $x$ has dimension $n$ then we say that $X$ has dimension $n$ at $x$.

\end{dfn}



\begin{theorem}\label{theorem1} Let $X$ be a compact metric space with dimension $n$ at each point. Then the following conditions are equivalent:
\begin{enumerate}
\item[a)] There is a map $f: X \to M$ from $X$ into a   $n$-manifold such that the induced homomorphism $H^{n}(f): H^n(M) \to H^n(X) $ is null and the set of multiple points $\mathcal{M}(f)$ is not dense in $X$;
\item[b)] $H^{n}(i): H^n(X) \to H^n(X \setminus int \, D^n)$ is an isomorphism for some $n$-disc $D^n \subset \mathcal{E}_n(X)$, where $i : X \setminus int \, D^n  \to X $ is the inclusion.
\end{enumerate}
\end{theorem}


\

Motivated by the above theorem, we define the following:

\begin{dfn} An $n$-dimensional space $X$ is called $n$-tight  if 
there is {\bf no} $n$-disc $D^n \subset \mathcal{E}_n(X)$ such that  $H^{n}(i):
H^n(X) \to H^n(X \setminus int \, D^n)$ is an isomorphism,  where $i : X
\setminus int \, D^n  \to X $ is the inclusion.

\end{dfn}

 Theorem \ref{theorem1} is a consequence of the following two lemmas.

\begin{lemma}\label{lemma1} Let $X$ be an $n$-dimensional compact metric space. If there is an $n$-disc $D^n \subset \mathcal{E}_n(X)$ such that $H^{n}(i): H^n(X) \to H^n(X \setminus int \, D^n)$ is an isomorphism then there is a retraction $f: X \to D^n$ such that $f(X \setminus int \, D^n) \subset \partial D^n $. 

\end{lemma}

\proof  Consider the following commutative diagram:


\begin{table}[h]
\centering
\begin{tabular}{c} \xymatrix{ H^n (X) \ar[d] & H^n(X, D^n) \ar[l]_{\simeq} \ar[d]_{\simeq} \\
H^n(X \setminus int \, D^n) & H^n(X \setminus int \, D^n , \partial D^n) \ar[l]}
\end{tabular}
\end{table}

\noindent where all homomorphisms are induced by inclusions. Note that the upper horizontal and the right arrows are isomorphisms. Since, by the assumption, the left arrow is also an isomorphism, it follows that the lower horizontal arrow is also an isomorphism. It follows that the coboundary homomorphism $H^{n-1}(\partial D^n) \to H^n(X \setminus int \, D^n, \partial D^n) $ is trivial. Thus, by the Hopf's extension theorem [see \cite{Engelking}, p. 94, (H)], there is a retraction $X \setminus int \, D^n \to \partial D^n$, which completes the proof. \qed

\begin{lemma}\label{lemma2} Let $X$ be a compact metric space of dimension $\geq  n$ at each point, and let $f :X \to M$ be a map into an  $n$-manifold $M$ such that $H^{n}(f): H^n(M) \to H^n(X) $ is null and the set of multiple points $\mathcal{M}(f)$ is not dense in $X$. Then there is an $n$-disc $D^n \subset \mathcal{E}_n(X)$ such that the homomorphism induced by the inclusion $H^{n}(i): H^n(X) \to H^n(X \setminus int \, D^n)$ is an isomorphism.
\end{lemma}
\proof Since $\mathcal{M}(f)$ is not dense in $X$, there is an open set $U \subset X$ such that $f$ maps homeomorphically $U$ onto $f(U)$ and $f^{-1}(f(U)) = U$. Since the dimension of $X$ at each point is $\geq n$, $dim \, U \geq n$. Consequently, $f (U )$ is an $n$-dimensional subset of $M$. By [\cite{Engelking}, Theorem 1.8.10, p. 76], the interior of $f(U)$ in $M$ is non-empty, so $f(U)$ contains an $n$-disc $B^n$. We set  $D^n := f^{-1}(B^n)$. Then $D^n \subset \mathcal{E}_n (X)$. We consider the following commutative diagram:

\begin{table}[h]
\centering
\begin{tabular}{c} \xymatrix{ H^{n-1} (\partial D^n) \ar[r]^{ \  \  \delta^\ast} & H^n(X) \ar[r]^{\hspace{-1.1cm}\vartheta} & H^n(X_D)\oplus H^n(D^n) \ar[r] & H^n(\partial D^n) \\
H^{n-1}(\partial  B^n) \ar[u] \ar[r] & H^n(M) \ar[u]^{H^{n}(f)} &  & }
\end{tabular}
\end{table}

\noindent where $X_D := X \setminus int \, D^n$,  the rows are parts of Mayer-Vietoris exact sequences, and the vertical homomorphisms are induced by the maps defined by $f$.

Since $H^n(\partial D^n) = 0$, $\vartheta$ is an epimorphism. Since $H^{n-1}(\partial B^n) \to H^{n-1}(\partial D^n)$ is an isomorphism and $H^{n}(f): H^n(M) \to H^n(X) $ is null, the homomorphism $\delta^\ast  : H^{n-1}(\partial D^n) \to H^n(X)$ is trivial, thus $\vartheta$ is a monomorphism. So, $\vartheta$ is an isomorphism. Consequently, since $H^n(D^n) = 0$, the homomorphism induced by the inclusion $H^n(X) \to H^n(X_D)$ is an isomorphism.  \qed


\

\noindent {\bf Proof of Theorem \ref{theorem1}:} From Lemma \ref{lemma1} we obtain that (b) implies (a). 
 From Lemma \ref{lemma2}, (a) implies (b).

\

\begin{obs}
 Let $X$ be an $n$-dimensional compact metric space, $C$ a component of $\mathcal{E}_n(X)$ and $D^n \subset C$ an $n$-disc. Then,  $H^{n}(i): H^n(X) \to H^n(X \setminus int \, D^n)$ is an isomorphism
if and only if $H^{n}(i_C): H^n(X) \to H^n(X \setminus C)$ is an isomorphism. 
Thus in the condition (b)  in Theorem~\ref{theorem1} and in the Definition~3.2,  $H^{n}(i)$ can be replaced by  $H^{n}(i_C)$.
\end{obs} 

\proof Since $dim \, X \leq n$, both $H^{n}(i)$ and $H^{n}(i_C)$ are epimorphisms. Thus we need to show that
$H^{n}(i): H^n(X) \to H^n(X \setminus int \, D^n)$ is a monomorphism
if and only if $H^{n}(i_C): H^n(X) \to H^n(X \setminus C)$ is a monomorphism,
or equivalently, that
$H^{n}(j): H^n(X, X \setminus int \, D^n) \to H^n(X)$ is trivial
if and only if $H^{n}(j_C): H^n(X, X \setminus C) \to H^n(X)$ is trivial.
The later equivalence is a consequence of the fact that $$ H^{n}(\iota):  H^n(X, X \setminus int \, D^n) \to H^n(X, X \setminus C)$$ is an epimorphism and that $H^{n}(j)= H^{n}(j_C) \circ H^{n}(\iota)$. \qed

\bigskip

\begin{corol} \label{corol1}
Let $X$ be an $n$-tight compact metric space with dimension $n$ at each point and  $M$  an  $n$-manifold.   
Then for every map $f: X \to M$  such that $H^{n}(f): H^n(M) \to H^n(X)$  is null the set of multiple points $\mathcal{M}(f)$ is dense in $X$.
In particular, if $X$ is a closed $n$-manifold and $H^{n}(f): H^n(M) \to H^n(X) $ is  null then the set of multiple points $\mathcal{M}(f)$ is dense in $X$.
\end{corol}

\


\begin{obs} In \cite{GonOre} it was proved that, given a map $f: M \to N$, where $M , N$ are manifolds of same dimension with $M$ closed, if the absolute degree of $f$ is different from $1$ then the set of multiple points of any map $g \in [f]$ is dense. Note that a closed manifold does not satisfy the condition (b) in Theorem \ref{theorem1}. Therefore, does not satisfy the condition (a).  So using   Theorem \ref{theorem1}, we conclude that for every map $f: M \to N$ such that  $H^{n}(f): H^n(N) \to H^n(M) $ is null then the set of multiple points $\mathcal{M}(f)$ is dense in $M$.

\end{obs}

\subsection{Construction of complexes that are not $n$-tight}

Here, we will analyze when a complex $K$ is not $n$-tight, i.e. when a complex $K$ satisfies the item (b) of theorem \ref{theorem1}. Let $K = K^{(n-1)} \cup e_1^n \cup \cdots \cup e_m^n$, where $K^{(n-1)}=S_1^{n-1} \vee  S_2^{n-1} \vee \cdots \vee S_k^{n-1}$ and the attaching map $ f_i: S_i \to S_1^{n-1} \vee  S_2^{n-1} \vee \cdots \vee S_k^{n-1}$ for the $n$-cell $e_i$ is given by the $k$-tuple of integers $(n_1^i, \ldots , n_k^i)$, $i=1, \ldots , m$. We want to determine if  $H^{n}(i): H^n(K) \to H^n(K  \setminus \stackrel{\circ}{D^n})$  is an isomorphism for some $n$-disc $D \subset \mathcal{E}_n(K)$. Note that it is the same than to determine if
$H^n(K \setminus \stackrel{\circ}{e}_i  \hookrightarrow K)  $ is an isomorphism for some $n$-cell $e_i$.

Let us focus on the $n$-cell $e_m$: Note that $$H^n(K) = \underbrace{\mathbb{Z} \oplus \cdots \oplus \mathbb{Z}}_{m-\text{times}} / \langle  \bar{n}_1 , \ldots , \bar{n}_k \rangle,$$ where $\bar{n}_1=(n_1^1, n_1^2, \ldots , n_1^{m-1}, n_1^m)$, $\bar{n}_2=(n_2^1, n_2^2, \ldots , n_2^{m-1}, n_2^m), \ldots , \\ \bar{n}_k=(n_k^1, n_k^2, \ldots , n_k^{m-1}, n_k^m)$, and $$H^n(K \setminus \stackrel{\circ}{e}_m) = \underbrace{\mathbb{Z} \oplus \cdots \oplus \mathbb{Z}}_{(m-1)-\text{times}} / \langle  \tilde{n}_1 , \ldots , \tilde{n}_k \rangle,$$ where $\tilde{n}_1=(n_1^1, n_1^2, \ldots , n_1^{m-1})$, $\tilde{n}_2=(n_2^1, n_2^2, \ldots , n_2^{m-1}), \ldots , \\\tilde{n}_k=(n_k^1, n_k^2, \ldots , n_k^{m-1})$. Moreover, the homomorphism $$ H^n(K) \to H^n (K \setminus \stackrel{\circ}{e}_m)$$ is given by the projection $$ \overline{ (a_1, \ldots , a_m)} \mapsto   \overline{ (a_1, \ldots , a_{m-1})}.$$ The above homomorphism is always onto and we want to characterize when it is also injective. 

\begin{prop} For the complex $K$ defined above, the homomorphism $H^n(K) \to H^n (K \setminus \stackrel{\circ}{e}_m)$  is injective iff there are integers $\beta_1, \ldots , \beta_k $ such that 
\begin{equation}\label{equation1}
\beta_1 \bar{n}_1 + \cdots + \beta_k \bar{n}_k = (0, \ldots , 0, 1).
\end{equation}

\end{prop}
\begin{proof} It is clear that the kernel of the homomorphism $H^n(K) \to H^n (K \setminus \stackrel{\circ}{e}_m)$ is generated by the class $\overline{ (0, \ldots , 0, 1)}$. Therefore $H^n(K) \to H^n (K \setminus \stackrel{\circ}{e}_m)$ is injective iff the class $\overline{ (0, \ldots , 0, 1)}$ is null in $\displaystyle H^n(K) = \underbrace{\mathbb{Z} \oplus \cdots \oplus \mathbb{Z}}_{m-\text{times}} / \langle  \bar{n}_1 , \ldots , \bar{n}_k \rangle$, which is equivalent to the existence of integers  $\beta_1, \ldots , \beta_k $ such that  $\beta_1 \bar{n}_1 + \cdots + \beta_k \bar{n}_k = (0, \ldots , 0, 1)$.

\end{proof}

\begin{obs} The equation (\ref{equation1}) has an integer solution iff we can get,  from the $m \times k$ matrix $$
\left[\begin{array}{cccc}
n_1^1&n_2^1&\ldots &n_k^1 \\
\vdots & \vdots & \vdots & \vdots \\
n_1^m & n_2^m & \ldots & n_k^m
\end{array}\right]
$$ by applying column elementary operations, a matrix such that the last column is of the form $
\left[\begin{array}{c}
0 \\
\vdots  \\
0\\
1
\end{array}\right]
$.

\end{obs}

\begin{corol} The complex $K$ defined above is not $n$-tight if there is $i\in \{1, \ldots , m\}$ such that the equation
\begin{equation}\label{equation2}
\beta_1 \bar{n}_1 + \cdots + \beta_k \bar{n}_k = (0, \ldots , 0, 1, 0, \ldots , 0)
\end{equation}
has integer solution.

\end{corol}

\section{Example of maps with the set of single points dense}

In \cite{GonOre},  it is constructed a map $g: D^2 \to \R^2$ such that the set of single points is dense in $D^2$ and $g$ is constant on the boundary of $D^2$.  Then, from $g$ it is obtained an example of a map $f: S^2 \to \R^2$ with the set of single points dense in $S^2$. 
We noticed  that the techniques used by S. Orevkov  in his construction works to construct such maps $g: D^n \to \R^2$ for $n \geq 2$ and, therefore, to construct  maps $f: S^n \to \R^2$, $n \geq 2$, with dense set of single points.  In case of Orevkov's example, a crucial step is the existence of a map $\varphi : D^2 \to A$, where $A=\{ z\in \R^2 \ | \ a \leq |z| \leq b\}$ is a $2$-dimensional annulus ($0<a<b$), such that $\varphi(D^2)=A$ and $\varphi$ mapping the interior of $D^2$ homeomorphically onto a dense open subset of $A$.

Below, we prove the equivalent of the above assertion for the $n$-dimensional unitary disc $D^n$ and an $n$-dimensional annulus.

\begin{lemma}\label{lemma4.1} Let $D^n$ be the $n$-dimensional disc and $A^n=S^{n-1} \times I$ the $n$-annulus. There is a map $\varphi: D^n \to A^n$ such that $\varphi(D^n)=A^n$ and $\varphi$ maps the interior of $D^n$ homeomorphically onto an open dense subset of $A^n$.
\end{lemma}
\begin{proof}
Let's identify the disc $D^n$ with $D^{n-1}\times I$, and let $p: D^{n-1} \to S^{n-1}$ be the map that collapses the boundary of $D^{n-1}$ into a point.  Then, the map $\varphi: D^{n-1} \times I \to S^{n-1}\times I$ given by $\varphi(x,t)=(p(x), t)$ satisfy the conditions of the lemma.

\end{proof}

\

The same type of example given by S. Orevkov holds for maps from $S^n$ into $\R^2$ for $n\geq 2$. 

\begin{theorem}\label{Orevkov} There is a map $g: D^n \to T \subset \R^2$ such that:
\begin{itemize}
\item $T$ is a fractal tree;
\item $g(S^{n-1})=\{0\}$; and
\item The set of single points of $g$ is dense in $D^n$.
\end{itemize}
\end{theorem}

\

\begin{corol} There is a map $f: S^n \to T \subset \R^2$ such that the set of single points of $f$ is dense in $S^n$.
\end{corol}
\begin{proof} The map $f: S^n \to \R^2$ is obtained from the map $g$ of the previous Theorem by collapsing the boundary of $D$ into a point since $g$ is constant on the boundary $\partial D$.

\end{proof}

\

More generally:

\begin{corol} Let $M^n$ be an $n$-dimensional connected manifold, $n \geq 2$. There is a map $f: M^n \to T \subset \R^2$ such that the set of single points of $f$ is dense in $M^n$.
\end{corol}
\begin{proof}From \cite{DyH}, $M=P^n \cup C$,   where $P^n$ is homeomorphic to the Euclidean $n$-space, $E^n$, and $C$ is a closed subset of $M^n$ of dimension at most $n-1$, and $P \cap C = \emptyset$. Let $h: P \to D^n - S^{n-1}$ be a homeomorphism and $g: D^n \to T \subset \R^2$ a map such that the set of single points of $g$ is dense in $D^n$ and $g(x)=0$ for all $x\in S^{n-1}$. Now, define $f: M \to T \subset \R^2$ by: $f(x)=g(h(x))$ if $x\in P$ and $f(x)=0$ if $x \in C$.

\end{proof}

\

\noindent {\it Proof of Theorem \ref{Orevkov}:} (Sketch) The construction of the map $g$ claimed in Theorem \ref{Orevkov} is obtained following the same steps as the map constructed by S. Orekov  in \cite[Appendix, Example]{GonOre}.  Consider  the following steps:

\begin{description}
\item[Step 1.] A decomposition of $D^n$ is defined;
\item[Step 2.] A particular tree $T \subset \R^2$ is considered;
\item[Step 3.] A uniformly Cauchy sequence of maps $g_m: D^n \to T$ is constructed in such way that, for each $m$, $g_m(\partial D)=0$;
\item[Step 4.] $g: D^n \to \R^2$ will be the limit of the sequence $\{g_m\}$.
\end{description}
\noindent where we briefly comment how to adapt the proof of the case $n=2$ for $n$ arbitrary.
\

The {\bf Step 1}(Construction of Annuli)  for $n=2$ uses  maps $\varphi: D^2\to  A_{\vec{b}}$ which satisfy 
\begin{enumerate}
\item $\varphi_{\vec{b}}(0)=p_{\vec{b}}$
\item $\varphi_{\vec{b}}(D^2)=A_{\vec{b}}$
\item $\varphi_{\vec{b}}(0)$ maps $IntD^2$ homeomorphically onto a dense open subset of  $A_{\vec{b}}$.
\end{enumerate}    

In our case such maps are replaced by the one given by the Lemma \ref{lemma4.1}, to obtain a similar decomposition 
using $n$-dimensional annulus. Also the analogue of Lemma 2.2, \cite{GonOre}, page 371, holds in our case, namely 
"$A$ and $P$ are dense in $D^n$" for the correspondent set $A$ in our case and $P$ is the same as in $n=2$.

The {\bf  Step 2},  the particular tree $T \subset \R^2$  considered is the same as in case $n=2$.
 For {\bf  Steps 3 and 4} the sequence of maps
are defined in a similar way as the maps $f_m$ as in the case $n=2$. So we obtain $g$ which is continuous, where the proof of the continuity is the same as for $f$ (see section 4. Construction of the Mapping  in  \cite{GonOre}, page 372).
This conclude the proof.
\qed

\

\begin{corol} \label{corol4.6} Let $X$ be an $n$-dimensional CW-complex with a finite number of $n$-cells, $e^n_1, \ldots , e^n_k$, attached to the $(n-1)$-skeleton $X^{(n-1)}$, $n \geq 2$. There is a map $f: X \to \R^2$ such that the set of single points of $f$ is dense in $X$.
\end{corol}
\begin{proof} Let $T_1, T_2, \ldots , T_k$ be copies of $T$ in $\R^2$ such that $T_i \cap T_j = \{ 0 \}$ for $i \neq j$. For the cell $e^n_i$, let $g_i : e^n_i \to T_i$ be a map such that $g_i$ restricts to the boundary $\partial e^n_i$  is constant equal zero, and the set of single points of $g_i$ is dense in $e^n_i$. Then, define $f: X \to \R^2$ by: $f(x)=g_i(x)$ if $x \in e^n_i$ and $f(x)=0$ if $x \in X^{(n-1)}$.
\end{proof}



It is natural to ask:   For which pairs of spaces $(X, \R^m)$, where $X$ is a finite CW-complex of dimension $n$ and $\R^m$ is the $m$-dimensional Euclidean space,   there exists  a map  $f:X \to \R^m$ such that  the set of single points of $f$ is  dense in $X$?  The examples above together with  few elementary remarks  show that we have a complete answer of the above question.  

   
   \begin{remark}
\begin{enumerate}   
\item  If $X$ is a graph, i.e., a CW-complex where  $dim(X)=1$, it is easy to see that such map $f: X \to \R$ exist if and only if $X$ is homeomorphic to the segment. This follows from the observation that the graph which is the letter $Y$ does not admit such map. 
   
   \item If $X$ is a graph, certainly there exist a such map $X \to \R^2$.  Take as image of $X$ a bouquet of circle. 
   
   \item  In the case  $m \geq 2$,   it suffices to construct such examples if $m=2$. Then the answer is yes,  and is obtained using the CW structure of $X$,  Corollary \ref{corol4.6} and item 2   above as follows: Write $X$ as the union (not disjoint union) of subcomplex where all  the maximal cell have the same dimension, then for each   such subcomplex apply either   Corollary \ref{corol4.6} or  item 2   above depending if the dimension of the subcomplex is $\geq 2$ or $1$, respectively. 
   
   \item There is {\bf no} map $f: D^2 \to \R$ for which the set of single points is dense in $D^2$. Indeed, let $f: D^2 \to \R$ be a map and, without loss of generality, suppose $0\in D^2$ a single point of $f$ and $f(0)=0 \in \R$. Then, $f(S^1)$ is a connected compact subset of $\R \setminus \{0\}$ and, therefore, it is entirely contained in one of the components of $\R \setminus \{0\}$, say in $\R^+$. Let $d$ be the distance of $0$ to $f(S^1)$. Thus, $d>0$. The subset $f^{-1}((-d/2, d/2))$ is an open subset of $D^2$ containing the origin. Let $\delta > 0$ such that the open ball $B(0, \delta)$ is contained in $f^{-1}((-d/2, d/2))$. We state that none point of $B(0, \delta) \setminus \{0\}$ is a single point of $f$. Let $z_0\in B(0, \delta)$, $z_0 \neq 0$. Then, $0<|f(z_0)|<d/2$. Note that $f(z_0)$ must be positive because the image of the annulus $\{z\in D^2 \ | \ |z_0| \leq |z| \leq 1\}$ is a connected compact subset of $\R \setminus \{0\} $. Now, take any radius in $D^2$ that does not contain the point $z_0$. The image of such radius is an interval with initial point $0$ and end point a number bigger then $d$. Hence, the image of some point belonging to this radius coincides with $f(z_0)$.  Therefore, $z_0$ is not a single point of $f$.
   
   \item Similarly, there is {\bf no} map $f: D^n \to \R$ for which the set of single points is dense in $D^n$, $n \geq 2$.

 \end{enumerate}

   \end{remark}

   \begin{remark}  In \cite{B}, it was shown that for $m , n \ge 3$ there is a monotone surjection $f: S^m \to S^n$ such that the set of single points of $f$ is  dense in $S^m$, where a map is  called
monotone provided each point-inverse is compact and connected. Moreover, it was shown that for $3 \le n \le 
   m \le 2n-3$, each element of $\pi_m(S^n) $ can be represented as a monotone surjection $f: S^m \to S^n$ with dense set of single points.

   Another natural question to ask is: For which pairs of spaces $(X, M)$ where $X$ is a finite CW-complex of dimension $n$ and $M$ is a $m$-dimensional manifold, there exists a surjective map $f: X \to M$ such that the set of single points of $f$ is  dense in $X$?  
   \end{remark}

 \
 
 \

\noindent Daciberg Lima Gon\c{c}alves\\
\noindent{\bf Department of Mathematics - IME\\
University of S\~ao Paulo\\
Rua~do~Mat\~ao~1010\\
~CEP:~05508-090 - S\~ao Paulo - SP - Brazil}\\
e-mail: \texttt{dlgoncal@ime.usp.br}

\

\noindent Tha\'is Fernanda Mendes Monis\\
\noindent{\bf Department of Mathematics - IGCE  \\
S\~ao Paulo State University (Unesp) \\
Av. 24 A ~ 1515 \\
~CEP:~13506-900 - Rio Claro - SP -  Brazil}\\
e-mail: \texttt{tfmonis@rc.unesp.br}

\

\noindent Stanis\l{a}w Spie\.z\\
\noindent{\bf Institute of Mathematics - Polish Academy of Sciences \\  IMPAN \\
ul. \'Sniadeckich ~8}\\
00-656 - Warsaw - Poland \\
e-mail: \texttt{spiez@impan.pl}

\end{document}